\documentclass[11pt]{amsart}
\input{epsf.tex}
\usepackage{euscript,color,latexsym,amsfonts,amssymb,epsfig,verbatim}
\usepackage{amsmath,amsthm,latexsym,graphics,textcomp}
\usepackage{url}
\usepackage{tikz}
\usepackage{graphicx}
\usepackage{psfrag, epsf}

\input xy
\xyoption{all}

\theoremstyle{plain}
\newtheorem{theorem}{Theorem}[section]
\newtheorem{lemma}[theorem]{Lemma}
\newtheorem{corollary}[theorem]{Corollary}

\newtheorem{proposition}[theorem]{Proposition}

\newtheorem{example}[theorem]{Example}

\begin{document}

\title{Totally geodesic surfaces with arbitrarily many compressions}

\author{Pradthana Jaipong}\thanks{}

\address{Department of Mathematics \\ University of Illinois at Urbana-Champaign \\ 
1409 W. Green st. \\
Urbana, IL 61801}

\email{jaipong2@illinois.edu}

\maketitle

\begin{abstract}
A closed totally geodesic surface  in the figure eight knot complement remains incompressible in all but finitely many Dehn fillings. In this paper, we show that there is no universal upper bound on the number of such fillings, independent of the surface. This answers a question of Ying-Qing Wu.
\end{abstract}

\section{Introduction} \label{introsect}
Let  $M$ be a compact, connected, irreducible, orientable 3-manifold with torus boundary $\partial M$. A slope on $\partial M$ is an isotopy class of simple closed curves on $\partial M$. We use $\Delta(\alpha, \beta)$ to denote the absolute value of the algebraic intersection number between the slopes $\alpha$ and $\beta$.    It is shown in \cite{culler} that if $F$ is a closed, orientable, embedded, incompressible surface in $M$ with no incompressible annulus joining $F$ and $\partial M$, and $F$ compresses in the Dehn fillings $M(\alpha)$ and $M(\beta)$, then $\Delta(\alpha, \beta) \leq 2$. In \cite{wuincompress}, Wu improved this to $\Delta(\alpha, \beta) \leq 1$, and hence $F$ remains incompressible in  $M(\gamma)$ for all but at most three slopes $\gamma$. 

If one drops the assumption that $F$ be embedded, the previous theorem is not true;  see \cite{leiningercompress}. However, for hyperbolic $M$ such a surface $F$ can compress in at most finitely many Dehn fillings $M(\gamma)$;  see \cite{bart}. In fact in \cite{wudepth}, it is shown that there is a bound on the number of fillings in which $F$ can compress depending only on the genus of $F$, and not on the manifold $M$. Wu has asked whether there is any universal bound, independent of $F$,  for this number (Question 6.6 in \cite{wuimmersed}). In this paper, we prove that no such universal bound exists. More precisely we prove

\begin{theorem} \label{mainthm}
There exists a compact, connected, orientable, 3-manifold $M$, with torus boundary and hyperbolic interior having the following properties. Given any positive integer n, there exist $n$ distinct slopes $\alpha_{1}, ... , \alpha_{n}$ and infinitely many pairwise non-commensurable closed, orientable, immersed, incompressible  surfaces $F \looparrowright M$, with no incompressible annulus joining $F$ and $\partial M$, such that F compresses in $M(\alpha_{i})$ for all  $i=1, ... ,n$.
\end{theorem}

The manifold in Theorem \ref{mainthm} is $M_{8}$, the exterior of the figure eight knot in $S^3$. Our proof involves a careful analysis of a construction of closed, immersed, totally geodesic surfaces in $M_{8}$ which compress in $M_{8} (\gamma)$ for some specific $\gamma$. In particular, we inspect the proof of the following theorem from \cite{leiningercompress}.

\begin{theorem} \label{chris}
Suppose $ 4 \mid p$ and $3\nmid p$. Then for any $q$ that is relatively prime to $p$  there exists infinitely many non-commensurable, closed, immersed, totally geodesic surfaces in $M_{8}$ which compress in $M_{8} ( {{p}\over {q}} )$.
\end{theorem}

This paper is organized as follows: Section 2 contains a few definitions and constructions from 3-dimensional topology necessary for our work. We then give a brief review of  some basic definitions and facts concerning hyperbolic 3-manifolds in Section 3. Section 4 contains the various constructions of totally geodesic surfaces and theorems on compressing totally geodesic surfaces. Section 5 includes definitions and facts from number theory and quadratic forms, and we prove the main technical theorem needed for the proof of Theorem \ref{mainthm}. In Section 6, we prove Theorem \ref{mainthm}. \\

\noindent \textbf{Acknowledgments.}  The author is indebted to Professor Christopher  Leininger for key discussions and  for his careful reading of early versions of this manuscript. The author would like to thank Professor Kenneth Williams for his helpful suggestion on using quadratic forms. The author also appreciates the support from NSF during the Spring and Summer 2010. \\

\section{3-dimensional Topology} \label{3d}
In this section we recall some definitions and facts from 3-dimensional topology. For more details, see  \cite{benedetti, rolfsen, thurston}.

Let $M$  be a compact, orientable 3-manifold with a torus boundary $\partial M \cong T^2$ and let  $\pi_{1} (\partial M) \cong \pi_{1}(T^2) \cong \mathbb{Z} \oplus \mathbb{Z}$ be generated by $\lambda$ and $\mu$. A \textit{slope} on $\partial M $ is an isotopy class of simple closed curves on $\partial M$, and can be uniquely associated (up to inverses) with a primitive element  $\alpha = \lambda^{p} \mu^{q} \in \pi_{1} (\partial M)$. Primitivity implies $p, q$ are relatively prime and so the set slopes are in a one-to-one correspondence with $\mathbb{Q} \bigcup \{\infty\}$, where $ \lambda^{p} \mu^{q}$ corresponds to ${{p} \over {q}}$ in the lowest terms. We write $\alpha = {{p} \over {q}}$ (with $\infty = {{1} \over {0}}$). If  $\alpha = {{p} \over {q}}$ and $\sigma = {{r} \over {s}}$ are two slopes in $ \pi_{1} (\partial M)$, then the \textit{distance} between $\alpha$ and $\sigma$ is given by   $\Delta (\alpha, \sigma) =  | ps-qr |$. 

Now let $\alpha$ be a slope on $\partial M$, $S^1 \times D^2$ be a solid torus and $\mu_{0}= \{*\} \times \partial D^2$ be a meridional curve on $\partial(S^1 \times D^2)$. We form a closed 3-manifold by $\alpha$-\textit{Dehn filling} on $\partial M$ by attaching $S^1 \times D^2$ to $M$ identifying $\partial (S^1 \times D^2)$ with $\partial M$ so that $\alpha$ is identified with $\mu_{0}$. The resulting space,  denoted by $M(\alpha)$, is a closed 3-manifold depending only on $\alpha$ up to homeomorphism.

Let $F$ be a closed, connected, orientable surface which is not homeomorphic to a 2-sphere. We say that an immersion $f:  F \longrightarrow M$  is an \textit{incompressible surface} if the induced map $f_{*} : \pi_{1}(F) \longrightarrow \pi_{1}(M)$ is injective, and \textit{compressible}, otherwise.  A surface $f : F \longrightarrow M$ is \textit{essential} if it is incompressible and is not homotopic into $\partial M$. We say  $f : F \longrightarrow M$  is \textit{acylindrical} if no element of $f_{*}(\pi_{1}(F))$ is \textit{peripheral}, that is, conjugate into $\pi_{1}(\partial M)$. Equivalently, there is no annulus in $M$ joining a non-trivial loop in $F$ to a loop in $\partial M$.

The manifold $M$ we are interested in here will be the interior of compact manifold $\overline{M}$ with torus boundary. We will write $M(\alpha)$ for $\overline{M}(\alpha)$, and will refer to $\partial M$ for $\partial \overline{M}$. We will generally not distinguish between $M$ and $\overline{M}$ when no confusion arises.

The figure eight knot $K \subset S^{3}$ is the knot whose projection is shown in Figure \ref{eight}. The manifold we will analyze is $M_{8} = S^{3}-K$ which is the interior of a compact manifold with torus boundary. In the next two sections, we will describe this manifold in more detail.

\begin{figure}[htb]
\begin{center}
\begin{tikzpicture}[scale=.7]
\draw [very thick] (0,.2) .. controls +(.5,0) and +(0,-.5) .. (1,1)
  .. controls +(0,.5) and +(.5,-.25) .. (0.1,1.95)
  (-.1,2.05) .. controls +(-.5,.25) and +(0,-.5) .. (-1,3)
  .. controls +(0,.5) and +(-.5,-.25) .. (0,4)
   .. controls +(3,1) 
      and +(4,0) .. (1.1,1)
   (.9,1) .. controls +(-4,0) and +(3,0) .. (-1.5,1)
   .. controls +(-3,0)
     and +(-3,1) .. (-.1,4.05)
   (.1,3.95) .. controls +(.5,-.25) and +(0,.5) .. (1,3)
   .. controls +(0,-.5) and +(.5,.25) .. (0,2)
   .. controls +(-.5,-.25) and +(0,.5) .. (-1,1.1)
   (-1,.9) .. controls +(0,-.5) and +(-.5,0) .. (0,.2)
   ;
\end{tikzpicture}
\caption{The figure eight knot.}\label{eight}
\end{center}
\end{figure}
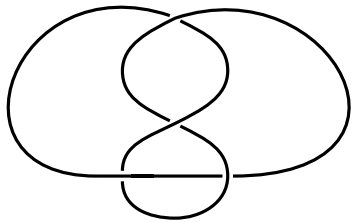

\section{Hyperbolic 3-manifolds} 
Here we review some of the background concerning hyperbolic 3-manifolds. See \cite{maclachlanarithmetic,  ratcliffe} for more details.

Let us  consider the upper half space $$\mathbb{H}^3 = \{ (z, t) \in \mathbb{C}\times \mathbb{R}  \; |  \; t > 0 \} $$  endowed with the complete Riemannian metric  
$$ds^2 = {{|dz|^2  + dt^2}\over{t^2}},$$ which is a model for hyperbolic $3$-space. The \textit{boundary at infinity} $\partial \mathbb{H}^3$ is $\mathbb{\widehat{C}} =  (\mathbb{C}\times \{0\}) \bigcup \{ \infty \}$. The group of all orientation-preserving isometries of $\mathbb{H}^3$ is isomorphic to  $PSL_{2} (\mathbb{C})$ acting by conformal extension of M\"obius transformation on $\mathbb{\widehat{C}}$.

 As a convention, since $PSL_{2} (\mathbb{C}) \cong SL_{2} (\mathbb{C}) /\pm I$, whenever we refer to a matrix for an element in $PSL_{2} (\mathbb{C}) $, we really mean one of the corresponding matrices in $SL_{2} (\mathbb{C})$ under the quotient homomorphism. A subgroup $\Gamma$ of  $PSL_{2} (\mathbb{C}) $  is said to be a \textit{Kleinian group} if the induced topology on $\Gamma$ is the discrete topology. Equivalently, $\Gamma$ acts properly discontinuously on $\mathbb{H}^3$.

  Throughout this paper, we will consider $\Gamma$ a torsion-free Kleinian group. Let $M_{\Gamma} = \mathbb{H}^3 / \Gamma$ be the quotient \textit{hyperbolic 3-manifold} with its induced metric, so we have $\pi_{1}(M_{\Gamma} ) \cong \Gamma$. We say that $\Gamma$ is \textit{co-compact} or has \textit{finite co-volume} if $M_{\Gamma}$ is compact or has finite total volume, respectively.   
   
 As is shown in \cite{riley} (see also  \cite{maclachlanarithmetic,  ratcliffe}),  $M_{8} \cong \mathbb{H}^3 / \Gamma_{8}$, and $\pi_{1}(M_{8}) \cong \Gamma_{8} $, 
where 
$$\Gamma_{8} = \left \langle \begin{pmatrix} 1 & 1 \\ 0 & 1 \end{pmatrix},    \begin{pmatrix} 1 & 0\\ -\omega & 1 \end{pmatrix} \right \rangle,$$
with $\omega = {{-1+ \sqrt{-3}} \over {2}}$. We note $\mathbb{Z} [\omega] = \mathcal{O}_{3}$, the ring of integers in the quadratic number field $ \mathbb{Q}(\sqrt{-3}) $. Furthermore $\Gamma_{8}$ has index 12 in  $PSL_{2} (\mathcal{O}_{3})$.

\section{Surfaces in hyperbolic 3-manifolds} 

In this section we collect some of definitions and facts concerning
 surface in hyperbolic 3-manifolds. See \cite{long,maclachlanarithmetic, thurston} for more details.

\subsection{Totally geodesic surfaces in  $M_{\Gamma} = \mathbb{H}^3 / \Gamma$ } 

All closed, orientable, immersed, totally geodesic surfaces in  $M_{\Gamma} = \mathbb{H}^3 / \Gamma$  arise as follows.

Let $\mathcal{C}$ be any circle in $\widehat{\mathbb{C}}$, that is, a circle or line in $\mathbb{C}$.  For any subgroup $\Gamma \subseteq PSL_{2} (\mathbb{C})$, define
 $$Stab_{\Gamma} (\mathcal{C}) = \{ g \in \Gamma \;  | \;  g (\mathcal{C}) = \mathcal{C} \text{ and } g  \text{ preserves the components of } \widehat{\mathbb{C}} \backslash \mathcal{C} \}.$$ 
 
For any circle $\mathcal{C}$, a discrete subgroup of $Stab_{PSL_{2} (\mathbb{C})} (\mathcal{C}) $ is called a \textit{Fuchsian group}. Because of the transitive action of $PSL_{2} (\mathbb{C})$  on circles in $\widehat{\mathbb{C}}$,  there exists $g \in PSL_{2} (\mathbb{C})$ such that $g(\mathcal{C}) = \widehat{\mathbb{R}} =  \mathbb{R} \bigcup \{\infty\}$, and hence $g Stab_{PSL_{2} (\mathbb{C})} (\mathcal{C}) g^{-1} = PSL_{2} (\mathbb{R})$. 

Any circle $\mathcal{C}$ in  $\widehat{\mathbb{C}}$ bounds a hyperbolic plane $P_{\mathcal{C}} \cong \mathbb{H}^2$ embedded totally geodesically in $\mathbb{H}^3$. If $\Gamma^{'} = Stab_{\Gamma} (\mathcal{C})$ is a torsion free Fuchsian group, we obtain a \textit{hyperbolic surface} $S_{\Gamma^{'}} =  P_{\mathcal{C}} / \Gamma^{'} $ with $\pi_{1} (S_{\Gamma^{'}}) \cong \Gamma^{'}$.

 Let $\Gamma$  be a finite co-volume torsion free Kleinian group such that there exists a circle $\mathcal{C} \subset \widehat{\mathbb{C}}$ for which $\Gamma^{'} = Stab_{\Gamma} (\mathcal{C})$ has finite co-area. One can check that this induces a proper totally geodesic incompressible immersion
 $$ S_{\Gamma^{'}} \cong  P_{\mathcal{C}}  / \Gamma^{'} \looparrowright M_{\Gamma}\cong \mathbb{H}^3 / \Gamma.$$ 

Let us write $\mathcal{C}_{D}$ to denote  a circle centered at the origin with radius $D \in \mathbb{Z^{+}}$. Consider the subgroup of $\Gamma_{8}$, 
$$\Gamma_{D} = Stab_{\Gamma_{8}} (\mathcal{C}_{D}) = \{\gamma \in \Gamma_{8}  \;|\;  \gamma (\mathcal{C}_{D}) = \mathcal{C}_{D} \},$$ which always has finite co-area \cite{maclachlanfuchsian}.

  We say $\Gamma_{D}$,  $\Gamma_{D^{'}}$ are  \textit{commensurable} if there exists an element $g \in \Gamma_{8}$ such that $(g\Gamma_{D}g^{-1}) \bigcap \Gamma_{D^{'}}$ is a finite index subgroup in both $g\Gamma_{D}g^{-1}$ and $\Gamma_{D^{'}}$. The following is a consequence of arithmeticity (see \cite{leiningercompress, reidthesis}).

\begin{theorem} \label{comm}
For each positive integer $D \equiv 2 \pmod{3}$, $\Gamma_{D}$ is a co-compact Fuchsian group and therefore $S_{D} = P_{D} / \Gamma_{D}  \looparrowright M_{8}$ is a closed totally geodesic surface, which is in particular acylindrical. Moreover, $\Gamma_{D}$ and $\Gamma_{D^{'}}$  are commensurable in $\Gamma_{8}$ if and only if $D = D^{'}.$
\end{theorem}

\subsection{Compressing totally geodesic surfaces}\label{surface}

For the proof of the main theorem, we analyze the construction used in the proof of Theorem \ref{chris}. The details of its proof are in \cite{leiningercompress}.

The construction is started by considering $\pi_{1}(\partial {\overline M_{8}}) \cong \mathbb{Z} \oplus \mathbb{Z}$ which is generated by the standard meridian-longitude $\mu$ and $\lambda$. 

For any given integers $p$ and $q$ with $\gcd(p,q) = 1$, set $\sigma = \lambda^{p} \mu^{q} \in  \pi_{1}(\partial \overline{M_{8}})$. Van Kampen's theorem implies that $$ \pi_{1}\left( M_{8} \left( {{p}\over {q}}\right) \right)= \Gamma_{8} / \langle\langle \sigma\rangle\rangle$$ where $\langle\langle \sigma\rangle\rangle$ is the normal closure of $\{\sigma\}$ in $ \Gamma_{8} $. 

Now suppose $ 4 \mid p$,  $3\nmid p$. For any positive integer $k$, construct integers $n_{k}$, $D_{k}$ as follows:

$$\begin{array}{lclcl}
n_{k} & = & n_{k} (p,q) & = &  -3(p^{2}+12q^2)(2+3k)+9 \\
D_{k} & = & D_{k} (p,q) & = &   (p^{2}+12q^2){(n_{k})} ^{2}+2+3k.
\end{array}$$

	Define  a sequence $\{\Gamma_{D_{k}}\}^{\infty}_{k=1}$ of pairwise non-commensurable, co-compact Fuchsian subgroups $\Gamma_{D_{k}} = Stab_{\Gamma_{8}} (\mathcal{C}_{D_{k}})$ of $\Gamma_{8}$. From Theorem \ref{comm} we obtain  the sequence $ \{ S_{D_{k}} = P_{D_{k}} /  \Gamma_{D_{k}} \looparrowright M_{8} \}^{\infty}_{k=1}$ of pairwise non-commensurable, closed, orientable, immersed, totally geodesic surfaces in $M_{8}$ with  $\pi_{1} (S_{D_{k}}) \cong \Gamma_{D_{k}}$. 

We can now restate a more precise version of Theorem \ref{chris}, the main theorem of \cite{leiningercompress}.

\begin{theorem} \label{modify}
For any integer $p$ such that $4\mid p$,  $3\nmid p$ and $q$ relatively prime to $p$, let  $ \{ D_{k} \}_{k=1} ^{\infty} = \{ D_{k} (p,q) \}_{k=1} ^{\infty} $ be as above. Then for every $k$, the closed, immersed, totally geodesic surface  $ S_{D_{k}} \looparrowright M_{8}$ compresses in $M_{8} ( {{p}\over {q}} )$.
\end{theorem}

We note that $D_{k} (p,q)$ depends only on $k$ and $p^{2}+12q^2$. Our approach to prove the main theorem is to show that for a given integer $n$, we can find at least $n$ ways to represent the form $ p^{2}+12q^2$, where $p, q$ satisfy the above hypothesis. More precisely, there exists a family $\{ (p_{i},q_{i}) \}_{i=1}^{m}$, where $m\geq n$, such that $ D_{k}( p_{i}, q_{i}) = D_{k}( p_{j}, q_{j}) $ for all ${i, j = 1,... , m}$ and all positive integer $k$. By Theorem \ref{modify}, there are infinitely many closed, immersed, totally geodesic surfaces  $ S_{D_{k}} \looparrowright M_{8}$ which compress in $M_{8} ( {{p_{i}}\over {q_{i}}} )$ for all $i=1,...,m$ and all positive integer $k$. To find such representations $\{ (p_{i},q_{i}) \}_{i=1}^{m}$, we will need some facts about quadratic forms.

\section{Quadratic forms} \label{quadraticsect}
In this section we recall the relevant facts from  number theory and basic properties of  Legendre symbol and quadratic forms  that will be important tools for the proof of the main theorem; (see \cite{Cox, lang} for more details).

For any integer $a$ and positive odd prime $p$,  the \textit{Legendre symbol} $\left( {{a}\over{p}}\right)$ is defined by

$$\left( {{a}\over{p}}\right)=
\begin{cases}  
\; 0, & \text{if }  p \mid a \\
+1, & \text{if } p \nmid a \text{ and there exists an integer $x$ such that $x^2 \equiv a\pmod{p}$} \\
-1, &  \text{otherwise}.
\end{cases}$$

We list here some well-known properties of the Legendre symbol we will need.

 \begin{proposition} \label{legendre}
 Let  $p, q$ be distinct, positive, odd primes, and $a, b$ be integers,
 \begin{enumerate}
  \item(Completely multiplicative law) $\left( {{ab}\over{p}} \right) = \left({{a}\over{p}}\right)\left({{b}\over{p}} \right). $
 \item (Quadratic reciprocity law)
 \subitem If $p \equiv 1 \pmod{4}$ or $q \equiv 1  \pmod{4}$ , then $\left({{p}\over{q}}\right)=\left({{q}\over{p}} \right)$.
  \subitem If $p \equiv q \equiv 3 \pmod{4}$, then $\left({{p}\over{q}}\right)= (-1) \left({{q}\over{p}} \right)$.
\item (First supplement to the quadratic reciprocity law) 
\subitem $\left( {{-1}\over{p}}\right)= 1$ \;  if and only if  \; $p \equiv 1\pmod{4}$.
\end{enumerate}
\end{proposition} 

Using this, we prove the following.

\begin{lemma}\label{three}
For any prime $p$ greater than $3$, $\left( {3} \over {p} \right) = 1$ if and only if $p \equiv 1 \text{ or } 11\pmod{12}$.
\end{lemma}

\begin{proof} Any prime $p$ greater than 3 has $p \equiv 1, 5, 7$ or $11 \pmod{12}$. Given such $p$ apply Proposition \ref{legendre} part (2), and the fact that for any integer $x$, $x^2 \equiv $ 0 or 1$\pmod{3}$.
\end{proof}

	 A quadratic form $f(x,y)=ax^2+bxy+cy^2$ is called \textit{primitive} if its coefficients, $a, b$ and $c$ are relatively prime. We say an integer $m$ is \textit{represented} by $f(x,y)$ if the equation $m=f(x,y)$ has an integer solution. If this solution has $x$ and $y$ relatively prime, then we say that $m$ is \textit{properly represented} by $(x,y)$. We declare two primitive forms $f(x,y)$ and $g(x,y)$ to be  \textit{properly equivalent}, and write  $f(x,y) \sim g(x,y)$, if there exist integers $ p, q, r $ and $s$ such that $f(x,y) = g(px+qy, rx+sy)$ and $ps-rq = 1$. One can easily check that this defines an equivalence relation.	 
	 
	 The \textit{discriminant} of the form $f(x,y)=ax^2+bxy+cy^2$ is $\mathcal{D} =b^2-4ac$. Direct computation shows that if $f(x,y) = g(px+qy, rx+sy)$ , then $\mathcal{D} _{f}=(ps-qr)^2 \mathcal{D}_{g}$, where $\mathcal{D}_{f}$ and $\mathcal{D}_{g}$ are the discriminants of the forms $f$ and $g$, respectively. This implies that properly equivalent forms have the same discriminant. We restrict our discussion only to the case $\mathcal{D} <0$, and then $f$ is \textit{positive definite}. A primitive positive definite form $f(x,y)=ax^2+bxy+cy^2$ is said to be a \textit{reduced form} if $|b| \leq a \leq c$, and if $|b|=a$ or $a=c$ then $b\geq 0 $.
	 
	 Each equivalence class has a good representative quadratic form by Lagrange's Theorem of  Reduced Forms.

\begin{theorem} Every primitive positive definite form is properly equivalent to a unique reduced form.
\end{theorem}

Note that for a fixed discriminant $\mathcal{D}<0$, there are only finitely many reduced forms. Therefore, the number of classes of primitive, positive definite forms of discriminant  $\mathcal{D} $ is finite. To see this, consider a reduced form $f(x,y)=ax^2+bxy+cy^2$. By definition of a reduced form, we have $b^2 \leq a^2, a \leq c$. This implies 
$$ -\mathcal{D}  = 4ac-b^2 \geq 4a^2 - a^2 = 3a^2, $$ and  
$$0 < a \leq \sqrt{{-\mathcal{D} } \over{3}}, \; \text{then} \; |b| \leq a \leq \sqrt{{-\mathcal{D} } \over{3}} \;.$$ Hence there are only finitely many of choices for the integers $a, b$ and $c$. For example we have the following.

\begin{lemma} \label{reduced} There are exactly two reduced forms of discriminant  $\mathcal{D} =-48$; namely, $3x^2 + 4y^2$ and $x^2+12y^2$.
\end{lemma}
\begin{proof}
From the discussion  above, a reduced form $f(x,y)=ax^2+bxy+cy^2$ of discriminant  $\mathcal{D} =-48$ must satisfies $|b| \leq a \leq \sqrt{{48} \over{3}} = 4$, and $0 \leq c \leq 16$. An explicit finite search reveal that $3x^2 + 4y^2$ and $x^2+12y^2$ are the only posibilities.
\end{proof}
	 
We will need the following theorem of Gauss.

\begin{theorem} \label{numberofrep1} 
Let $m$ be a positive odd number relatively prime to $k > 1$. Then the number of ways that m is properly represented by a reduced form of discriminant $-4k$ is
 $$2 \prod_{p\mid m} \left(1+ \left({{-k}\over{p}}\right) \right),$$ where the product is over all distinct positive prime divisors $p$ of $m$. 
\end{theorem}

This theorem allows us to prove the following.

\begin{theorem}\label{numberofrep2}
For a given positive integer $m \equiv 7 \pmod{12}$ with all prime divisors congruent to  $1$ or $7 \pmod{12}$, the number of proper representations of $m$ by the primitive positive form $3x^2 + 4y^2$ is $2^{\tau(m) + 1}$ where $\tau (m)$ is the number of positive prime divisors of $m$.
\end{theorem}

\begin{proof}
Let us first investigate the properties of $m$.  Observe that $m \equiv 7 \pmod{12}$ implies that $m$ is odd and $m \equiv 3\pmod{4}$. Since 3 is not a square $\pmod{4}$, $m$ cannot be properly represented by the form  $x^2+12y^2$. 

With the given conditions on the divisors of  $m$, we can write 
$$m = {p_{1}}^{\alpha_{1}}\cdots{p_{s}}^{\alpha_{s}}{q_{1}}^{\beta_{1}}\cdots{q_{t}}^{\beta_{t}}$$ where $ {p_{1}},... ,{p_{s}}$ are distinct positive primes congruent to $1 \pmod{12}$  and ${q_{1}},... ,{q_{t}}$ are distinct positive primes congruent to $7 \pmod{12}$. Then
 $$ 7 \equiv m \equiv 7^{\beta_{1} + \cdots + \beta_{t}} \pmod{12}$$ 
 and it follows that 
 $${\beta_{1} + \cdots + \beta_{t}  \equiv 1 \pmod{2}}.$$

Now we consider the proper equivalence classes of the fixed discriminant $-48$. By Lemma \ref{reduced}, there are exactly 2 classes of primitive, positive definite reduced forms $3x^2 + 4y^2$ and $x^2+12y^2$. However, as noted above, $m$ cannot be represented by the latter form. Therefore, appealing to Theorem \ref{numberofrep1},  Proposition \ref{legendre} and Lemma \ref{three}, the number of ways that $m$ is properly represented by a reduced form  $3x^2 + 4y^2$ is
\begin{eqnarray*}
 2 \prod_{p\mid m} \left(1+ \left({{-12}\over{p}}\right) \right)
 &=&2 \prod_{p\mid m} \left(1+ \left({{-3}\over{p}}\right)\left({{4}\over{p}}\right) \right) \\
&=& 2 \prod^{s}_{i=1} \left(1+ \left({{-3}\over{p_{i}}}\right) \right) \prod^{t}_{i=1} \left(1+ \left({{-3}\over{q_{i}}}\right) \right)\\
  &=& 2 \prod^{s}_{i=1} \left(1+ \left({{-1}\over{p_{i}}}\right)\left({{3}\over{p_{i}}}\right) \right) \prod^{t}_{i=1} \left(1+ \left({{-1}\over{q_{i}}}\right) \left({{3}\over{q_{i}}}\right) \right)\\
   &=& 2 \prod^{s}_{i=1} \left(1+ (1)(1)\right) \prod^{t}_{i=1} \left(1+ (-1)(-1)\right)\\
 &=& 2^{(s+t)+1} = 2^{\tau(m) + 1}.
\end{eqnarray*}
\end{proof}

 As a consequence, we have the following corollary.

\begin{corollary}\label{numberofrep3}
Let $N=4m$ for some positive integer $m$ such that $m \equiv 7 \pmod{12}$ and all prime divisors of $m$ are congruent to $1$ or $7$ $\pmod{12}$. The number of ways to properly represent $N$ in the form $N = p^{2}+12q^2$, where $4 \mid p$ and $ 3 \nmid p$, is exactly $2^{\tau(N)}$.
\end{corollary}

\begin{proof} Writing $p = 4r$ sets up to bijection between the proper representations of $N = p^{2}+12q^2$ and $ m = 3q^2 + 4r^2$. So it suffices to prove that the number of ways to properly represent $ m = 3q^2 + 4r^2$ is $2^{\tau(N)} = 2^{\tau(m) + 1} $.

Applying Theorem \ref{numberofrep2}, the number of ways to represent $ m = 3q^2 + 4r^2$ is exactly $2^{\tau(m) +1}$, as  required.
\end{proof}

\section{Compression}

In this section we prove 

\begin{theorem} \label{semimain}
Given any positive integer $n$,  there exist $n$ distinct slopes $\alpha_{1}, ... , \alpha_{n}$ in the $\partial M_{8}$ and infinitely many closed, orientable, immersed, incompressible  surfaces $S_{D_{k}}\looparrowright M_{8}$ with no incompressible annulus joining $S_{D_{k}}$ and $\partial M_{8}$ which  compress in $M_{8}(\alpha_{i})$ for all  $i=1, ... ,n$ and positive integer $k$.
\end{theorem}

From this, Theorem \ref{mainthm} easily holds.

\begin{proof}[Proof of Theorem \ref{mainthm}]
By assuming Theorem \ref{semimain}, this theorem immediately follows when we let $M=\overline M_{8}$ which is  a compact, orientable, irreducible 3-manifold with torus boundary. We note that $M(\alpha)=M_{8}(\alpha)  $ for any slope $\alpha$ in $\partial M$.

\end{proof}

To prove Theorem \ref{semimain} we first recall Dirichlet's Theorem on arithmetic progressions (see \cite {selberg} for more details).

 \begin{theorem}
 If positive numbers $s$ and $t$ are relatively prime, then there are infinitely many primes $p$ such that $p\equiv s \pmod{t}$.
 \end{theorem}
 
 Using this, we can prove 
 
\begin{lemma} \label{anyn}
For any given positive integer $n$, there exists a family of $n$ pairs $\{ (p_{i},q_{i}) \}_{i=1}^{n}$ such that $ p_{i}^{2}+12q_i^2 = p_{j}^{2}+12q_j^2$ for all $ i,j = 1,..., n $, where $p_{i}$  and $q_{i}$ are relatively prime, $4 \mid p_{i}$, but $ 3 \nmid p_{i}$ for all ${i = 1, ... , n}$.
\end{lemma}

\begin{proof} For any given positive integer $n$, there exists a positive integer $k$ such that $n \leq 2^k$. Define the integer $N$ by
$$N=4{a_{1}}^{\beta_{1}}\cdots{a_{k-1}}^{\beta_{k-1}},$$ 
where $ {a_{1}},... ,{a_{k-1}}$ are distinct  positive primes congruent to $7 \pmod{12}$ and $\beta_{1} + \cdots + \beta_{k-1} \equiv  1 \pmod{2}$. Since 7 and 12 are relatively prime, we know that such integer $N$ exists by Dirichlet's theorem. 

By construction, $N$ satisfies the hypothesis of  Collorary  \ref{numberofrep3}. Therefore, there exist $2^{\tau(N)}= 2^{k} \geq n$ pairs $(p, q)$ which properly represent $N$, and moreover these satisfy the conditions of the lemma.

\end{proof}

\begin{example} Using Mathematica for $n=16$, we have the family  

$\{(32, 813), (200, 811), (680, 789), (1112, 747), (1328, 717), (1528, 683), \\ (1640, 661), (1912, 597), (2032, 563), (2320, 461), (2560, 339), (2608, 307), \\ (2648, 277), (2720, 211), (2752, 173), (2792, 107)\}.$ 

Each such pairs $(p, q)$ are relatively prime, $4 \mid p$, $3 \nmid p$ and $p^{2}+12q^2= 7,932,652$. 
\end{example}

\begin{proof}[Proof of Theorem \ref{semimain}]
For any given $n$, by Lemma \ref{anyn} there exists a family  $\{ (p_{i},q_{i}) \}_{i=1}^{n}$ such that $ p_{i}^{2}+12q_i^2 = p_{j}^{2}+12q_j^2$ for all $ i,j = 1,..., n$, where $p_{i}$  and $q_{i}$ are relatively prime such that $4 \mid p_{i}$, but $ 3 \nmid p_{i}$ for all ${i = 1, ... , n}$. For each ${i = 1, ... , n}$, consider the slope $\alpha_{i} = {{p_{i}}\over {q_{i}}} $ on $\partial M_{8}$. As noted in Section \ref{surface}, $ D_{k}( p_{i}, q_{i}) = D_{k}( p_{j}, q_{j}) $ for all ${i, j = 1,... , n}$ and all $k > 0$, and we denote this simply as $D_{k}$. Theorem \ref{modify} implies $S_{D_{k}}$ compresses in $M_{8} (\alpha_{i})$ for all ${i = 1,... , n}$ and all $k > 0$. Since $D_{1} < D_{2} < ...$, Theorem \ref{comm} implies  $S_{D_{k}}$ are all non-commensurable.

\end{proof}

 \bibliographystyle{plain}
\bibliography{multitube}

\def\cprime{$'$}
\begin{thebibliography}{10}

\bibitem{bart}
Anneke Bart.
\newblock Surface groups in some surgered manifolds.
\newblock {\em Topology}, 40(1):197--211, 2001.

\bibitem{benedetti}
Riccardo Benedetti and Carlo Petronio.
\newblock {\em Lectures on hyperbolic geometry}.
\newblock Universitext. Springer-Verlag, Berlin, 1992.

\bibitem{Cox}
David~A. Cox.
\newblock {\em Primes of the form {$x^2 + ny^2$}}.
\newblock A Wiley-Interscience Publication. John Wiley \& Sons Inc., New York,
  1989.
\newblock Fermat, class field theory and complex multiplication.

\bibitem{culler}
Marc Culler, C.~McA. Gordon, J.~Luecke, and Peter~B. Shalen.
\newblock Dehn surgery on knots.
\newblock {\em Ann. of Math. (2)}, 125(2):237--300, 1987.

\bibitem{lang}
Serge Lang.
\newblock {\em Algebraic number theory}, volume 110 of {\em Graduate Texts in
  Mathematics}.
\newblock Springer-Verlag, New York, second edition, 1994.

\bibitem{leiningercompress}
C.~J. Leininger.
\newblock Compressing totally geodesic surfaces.
\newblock {\em Topology Appl.}, 118(3):309--328, 2002.

\bibitem{long}
Darren Long and Alan~W. Reid.
\newblock {\em Surface subgroups and subgroup separability in 3-manifold
  topology}.
\newblock Publica\c c\~oes Matem\'aticas do IMPA. [IMPA Mathematical
  Publications]. Instituto Nacional de Matem\'atica Pura e Aplicada (IMPA), Rio
  de Janeiro, 2005.
\newblock 25${^{{}}{\rm{o}}}$ Col{\'o}quio Brasileiro de Matem{\'a}tica. [25th
  Brazilian Mathematics Colloquium].

\bibitem{maclachlanfuchsian}
C.~Maclachlan.
\newblock Fuchsian subgroups of the groups {${\rm PSL}_2(O_d)$}.
\newblock In {\em Low-dimensional topology and {K}leinian groups
  ({C}oventry/{D}urham, 1984)}, volume 112 of {\em London Math. Soc. Lecture
  Note Ser.}, pages 305--311. Cambridge Univ. Press, Cambridge, 1986.

\bibitem{maclachlanarithmetic}
Colin Maclachlan and Alan~W. Reid.
\newblock {\em The arithmetic of hyperbolic 3-manifolds}, volume 219 of {\em
  Graduate Texts in Mathematics}.
\newblock Springer-Verlag, New York, 2003.

\bibitem{ratcliffe}
John~G. Ratcliffe.
\newblock {\em Foundations of hyperbolic manifolds}, volume 149 of {\em
  Graduate Texts in Mathematics}.
\newblock Springer, New York, second edition, 2006.

\bibitem{reidthesis}
A.~W. Reid.
\newblock Ph.d. thesis.
\newblock Univ. Aberdeen, 1987.

\bibitem{riley}
Robert Riley.
\newblock A quadratic parabolic group.
\newblock {\em Math. Proc. Cambridge Philos. Soc.}, 77:281--288, 1975.

\bibitem{rolfsen}
Dale Rolfsen.
\newblock {\em Knots and links}, volume~7 of {\em Mathematics Lecture Series}.
\newblock Publish or Perish Inc., Houston, TX, 1990.
\newblock Corrected reprint of the 1976 original.

\bibitem{selberg}
Atle Selberg.
\newblock An elementary proof of {D}irichlet's theorem about primes in an
  arithmetic progression.
\newblock {\em Ann. of Math. (2)}, 50:297--304, 1949.

\bibitem{thurston}
W.~P. Thurston.
\newblock The geometry and topology of $3$--manifolds.
\newblock Princeton Lecture Notes, 1979.

\bibitem{wuincompress}
Ying~Qing Wu.
\newblock Incompressibility of surfaces in surgered {$3$}-manifolds.
\newblock {\em Topology}, 31(2):271--279, 1992.

\bibitem{wuimmersed}
Ying-Qing Wu.
\newblock Immersed essential surfaces and {D}ehn surgery.
\newblock {\em Topology}, 43(2):319--342, 2004.

\bibitem{wudepth}
Ying-Qing Wu.
\newblock Depth of pleated surfaces in toroidal cusps of hyperbolic
  3-manifolds.
\newblock {\em Algebr. Geom. Topol.}, 9(4):2175--2189, 2009.

\end{thebibliography}

\end{document}